\documentclass[11pt,reqno,a4paper]{amsart}

\usepackage{a4wide}
\usepackage[english]{babel}              
\usepackage[latin1]{inputenc}
\usepackage{amscd}
\usepackage{amsfonts}
\usepackage{amsgen}
\usepackage{amsmath}
\usepackage{amssymb}
\usepackage{amstext}
\usepackage{amsthm}
\usepackage{graphicx}
\usepackage{indentfirst}
\usepackage{latexsym}
\usepackage{epsfig}
\usepackage{wrapfig}
\usepackage[all,knot,arc,import,poly]{xy}
\usepackage[usenames]{color}
\usepackage[mathscr]{eucal}
\usepackage{ dsfont } 

\newtheorem{Df}{Definition}[section]
\newtheorem{Th}[Df]{Theorem}
\newtheorem{Prop}[Df]{Proposition}
\newtheorem{Lem}[Df]{Lemma}

\newtheorem{Rmk}[Df]{Remark}

\newtheorem{Cor}[Df]{Corollary}
\newtheorem{Ass}[Df]{Assumption}

\newcommand{\bc}{\begin{center}}
\newcommand{\ec}{\end{center}}

\newcommand{\ccal}{\mathcal{C}}

\newcommand{\jcal}{\mathcal{J}}
\newcommand{\lcal}{\mathcal{L}}
\newcommand{\pcal}{\mathcal{P}}

\newcommand{\nds}{\mathds{N}}
\newcommand{\rds}{\mathds{R}}

\newcommand{\Ll}{\Lambda(\lambda)}
\newcommand{\Lxx}{\Lambda^{*}(x)}
\newcommand{\lt}{\left}
\newcommand{\rt}{\right}

\newcommand{\ds}{\displaystyle}

\title{Large deviation for return times}
\author{Adriana Coutinho, J\'er\^ome Rousseau, Beno\^it Saussol}
\thanks{This work was partially supported by CAPES, CNPq and FAPESB}
\begin{document}

\maketitle

\begin{abstract}
  We prove a large deviation result for return times of the orbits of a dynamical system in a $r$-neighbourhood of an initial point $x$.
  Our result may be seen as a differentiable version of the work by Jain and Bansal who considered the return time of a stationary and ergodic process defined in a space
  of infinite sequences.
\end{abstract}


\section{Introduction}
Consider a dynamical system $(X, \mathcal{A}, g, \mu)$ where $X$ is a compact metric space, $\mathcal{A}$ is a $\sigma$-algebra on $X$, $g:X\rightarrow X$ is a
measurable map and $\mu$ an invariant probability measure on $(X, \mathcal{A}).$
Let $A \subset X$ be a measurable set of positive measure. An important result in ergodic theory is the Poincar\'e's recurrence theorem. It states that any
probability measure preserving map has almost everywhere recurrence. More precisely, for $\mu$-almost every $x \in A$ we have that $\{n: g^nx \in A\}$ is infinite.
It is natural to ask for more quantitative results of the recurrence. Given a point $x \in A$, the first return time of the orbit of $x$ to the set $A$ is given by
$$\tau_A(x)= \min\lt\{n \geq 1: g^nx \in A\rt\}.$$
 In \cite{Kac}, Kac has proven that, when the system is ergodic,
 $$\int_A \tau_A d\mu=1.$$
In other words, the average of the return time in a set $A$ is equal to the inverse of the measure of this set.
This subject have been further studied by many authors.
Boshernitzan \cite{Mbos} has established a link between the Hausdorff dimension of $X$ and the time needed by an orbit to approach its initial point.
To review results on recurrence see \cite{Saussol} and references therein. The author considers an expanding map of the interval and proves results for recurrence rates, limiting distributions of return times, and short returns. In \cite{Galvess} was presented an upper bound for the exponential approximation of the law of a hitting time in a mixing dynamical system.

Several works already addressed large deviations for return time.
Abadi and Vaienti in \cite{Abadivaiente} proved large deviation properties of $\tau(C_n)/n,$ where
$\tau(C_n)$ is the first return of a n-cylinder to itself. More precisely, if the
system is $\psi$-mixing, if $\psi(0)<1$ and the R\'enyi entropies exist for all integers $\beta$, then for $\delta \in (0,1],$ the limit
$$\lim_{n \to \infty}\frac{1}{n}\log \mu\lt\{x: \tau(C_n) \leq [\delta n]\rt\}:=M(\delta)$$
exists. In addition, they explicit the form of $M(\delta)$.

A large deviation result for the $n$th return times $\tau_A^n(x)$ into a fixed set $A$ we also considered by
Chazottes and Leplaideur \cite{ChLe} (See also \cite{LeSa}).
The Birkhoff theorem gives that for $\mu$-almost every point $x$
$$\lim_{n \to \infty} \frac{\tau^n_A(x)}{n}=\frac{1}{\mu(A)}.$$
For Axiom A diffeomorphisms and equilibrium states $\mu$, they prove the existence of a rate function $\Phi_A$, such that for every $u \geq \frac{1}{\mu(A)},$
$$\lim_{n \to \infty} \frac{1}{n} \log \mu \lt\{\frac{\tau_A^n}{n} \geq u\rt\}=\Phi_A(u),$$
with the appropriate change in the definition when $0 \leq u \leq \frac{1}{\mu(A)}$.

Our result concerns a different notion of return time, and may be seen as a differentiable version of a recent work by Jain and Bansal \cite{Bansal}. They studied large deviation property for repetition times under $\phi$-mixing conditions. Let $H$ denote the entropy rate of a finite-valued process $X=(X_n)$ and $x$ a particular realization of $X$. Define the first return time of
$x_1^n$ as $$R_n(x)=\min \lt\{j \geq 1: x_1^n=x_{-j+1}^{-j+n}\rt\}.$$ We say that $X$ have exponential rates for entropy if for every $\epsilon>0,$ we have
$$P\lt(\lt\{x_1^n:2^{-n(H+\epsilon)} \leq P(x_1^n) \leq 2^{-n(H-\epsilon)}\rt\}\rt)\leq 1-r(n,\epsilon),$$ where $r(\epsilon,n)=e^{-k(\epsilon)n},$ with
$k(\epsilon)$ a real valued positive function of $\epsilon.$ They proved that for an exponentially $\phi$-mixing process with exponential rates for entropy,
$$P\lt(\lt|\frac{\log R_n(X)}{n}-H\rt|> \epsilon\rt)\leq 2e^{-I(\epsilon)n}\ \mbox{for any $n$ sufficiently large.}$$ Where $I(\epsilon)$ is a real positive valued function for all $\epsilon>0$ and $I(0)=0.$

In the study of quantitative recurrence, an object of investigation is the return time $\tau_r(x)$ of a point  $x\in X$ under the map $g$ in its $r$-neighborhood, defined as follows:
$$\tau_r(x)=\tau_{B(x,r)}(x)= \min\{n \geq1: d(g^nx,x)<r\}.$$
If we denote with
$$\underline{d}_{\mu}(x)=\underset{r \rightarrow 0}{\underline{\lim}} \frac{\log \mu(B(x,r))}{\log r} \ \ \mbox{and} \ \
\overline{d}_{\mu}(x)=\underset{r \rightarrow 0}{\overline{\lim}} \frac{\log \mu(B(x,r))}{\log r},$$
the lower and upper pointwise dimensions of the measure $\mu$ at the point $x \in X,$ it was proved by Barreira and Saussol \cite{BarS} that
$$\underset{r \rightarrow 0}{\underline{\lim}} \frac{\log \tau_r(x)}{- \log r} \leq \underline{d}_\mu(x) \ \ \mbox{and} \ \
 \underset{r \rightarrow 0}{\overline{\lim}} \frac{\log \tau_r(x)}{- \log r} \leq \overline{d}_\mu(x), $$
for $\mu$-almost every $x \in X.$ If the system has a super-polynomial decay of correlations, Saussol in \cite{Saussol1} showed that equalities will hold for the expressions above. This implies that
$$\log \tau_r(x) \underset{r \rightarrow 0}{\sim} \log \lt(r^{-d_{\mu}(x)}\rt).$$

Our aim is to study the limiting behavior as $r \to 0$ of
$\mu\lt(\tau_r \geq r^{-d_{\mu}- \epsilon}\rt)$ and $\mu\lt(\tau_r \leq r^{-d_{\mu}+ \epsilon}\rt).$ This characterization is via asymptotic exponential
bound. We consider the limits
$$\underset{r \rightarrow 0}{\underline{\lim}} \frac{1}{\log r} \log \mu\lt(\tau_r \geq r^{-d_{\mu}- \epsilon}\rt) \ \mbox{and} \
\underset{r \rightarrow 0}{\underline{\lim}} \frac{1}{\log r} \log \mu\lt(\tau_r \leq r^{-d_{\mu}+ \epsilon}\rt).$$
Large deviations results are often related to multifractal analysis \cite{Pesin}. It turns out that
in the case of conformal repellers, the multifractal spectra is degenerate \cite{SaWu,FengWu}, that is
$$\dim_H\lt\{x \in X: \underset{r \to 0}{{\lim}} \frac{\log \tau_r(x)}{-\log r}=\alpha\rt\}=\dim_H X$$
for any $0\le\alpha \leq \infty$. It is not clear if this influences large deviations for return time.

This work is organized as follows. In Section \ref{generalresult} we define rates functions for dimension and for fast return times and state the general result, which is proven in Section \ref{proofmainresult}. In Section \ref{aprepeller} we consider a $C^{1+\alpha}$ conformal repeller and an equilibrium state of a H\"older potential. Then, we compute the rates functions and applies the main theorem to obtain large deviation estimates for return times for repeller.

\section{Large deviation estimates for return times in a general setting}\label{generalresult}

Let $g:X\rightarrow X$ be a measurable map and $\mu$ an invariant probability measure on $(X, \mathcal{A}).$

\begin{Df}
The measure $\mu$ is called exact dimensional if there exists a constant $d_\mu$ such that
$$\underline{d}_{\mu}(x)=\overline{d}_{\mu}(x)=d_{\mu} \ \ \mbox{for} \ \ \mu\mbox{-almost every} \ \ x \in X.$$
\end{Df}
We recall that the Hausdorff dimension of a probability measure $\mu$ on $X$ is given by
$$\dim_H \mu=\inf \{\dim_H Z: \mu(Z)=1\},$$
where $\dim_H Z$ denotes the Hausdorff dimension of $Z.$

Moreover, for an exact dimensional measure, the Hausdorff dimension and the local dimension coincide:
\begin{Prop} [\cite{Young}]
If $\mu$ is exact dimensional, then
$$d_{\mu}=\dim_H \mu.$$
\end{Prop}
We now define the rates functions which will appear in our large deviations estimates.
The first one is related to the deviations in the pointwise dimension; it has been computed in \cite{Pesin} in the case of conformal repellers.
\begin{Df} The exponential rate for dimension is defined for $\epsilon>0$ by:
\begin{equation} \label{expratedim}
\underline{\psi}(\pm\epsilon)=\underset{r \rightarrow 0}{\underline{\lim}}\frac{1}{\log r}\log \mu \left(\left\{ \frac{\log \mu(B(x,r)) }{-\log r} \in I_{\pm\epsilon} \right\}\right),
\end{equation}
where $I_{\epsilon}=(-\infty, -d_{\mu} - \epsilon)$ and $I_{-\epsilon}=(-d_{\mu}+\epsilon, +\infty).$
\end{Df}
The second quantifies the probability  of quick returns near the origin.
\begin{Df} The exponential rate for fast return times is defined for $\epsilon, a>0$ by:

\begin{equation}\label{exprateretur}
\underline{\varphi}(a,\epsilon)=\underset{r \rightarrow 0}{\underline{\lim}} \frac{1}{\log r} \log \mu\lt(\lt\{x_0: \mu_{B(x_0,2r)}\lt(\tau_{B(x_0,2r)}
\leq r^{-d_{\mu}+\epsilon}\rt) \geq Cr^a\rt\}\rt), \ 
\end{equation}
for some constant $C>0$.
\end{Df}

We may now state our main result. We emphasize that the value of $C$ in \eqref{exprateretur} is irrelevant in the theorem.

\begin{Th}\label{principalth}
Let $(X, \mathcal{A}, g, \mu)$ be a dynamical system. Suppose that $\mu$ is an exact dimensional measure. Given $\epsilon>0,$ we have:
\begin{eqnarray}\label{f1}
\underset{r \rightarrow 0}{\underline{\lim}} \frac{1}{\log r} \log \mu\lt(\tau_r \geq r^{-d_{\mu}- \epsilon}\rt) &\geq &
\max_{\gamma \in (0,1)}\min\left\{(1-\gamma)\epsilon, \underline{\psi}(\gamma\epsilon)\right\} \\
  \label{f2}
  \underset{r \rightarrow 0}{\underline{\lim}} \frac{1}{\log r} \log \mu\lt(\tau_r \leq r^{-d_{\mu}+ \epsilon}\rt)
  &\geq& \max_{\substack {\gamma \in (0,1) \\ a,\epsilon''>0}}\min \left\{-\gamma\epsilon-\epsilon''+a, \underline{\psi}(\gamma\epsilon),\underline{\varphi}(a,\epsilon), \underline{\psi}(-\epsilon'')\right\}.
  \end{eqnarray}
\end{Th}

This result is satisfactory in the sense that it can be applied in a broad class of dynamical systems,
provided one can estimate the rate functions $\underline{\psi}$ and $\underline{\varphi}$.

The rate function for dimension $\underline{\psi}$ is rather classical.
We can observe that in \eqref{f1} if the rate function for dimension
$\underline{\psi}$ is positive in some interval $(0,\epsilon),$ it readily implies that $\mu\lt(\tau_r \geq r^{-d_{\mu}-\epsilon}\rt)$ has a fast decay.

The rate function $\underline{\varphi}$ is not so well known. However, for several dynamical systems an estimation of the error in the approximation to the exponential law for return time has been computed. In many cases, including H\'enon maps \cite{ChCo}, it is possible to show that for some $a,b>0$, and any sufficiently small $r>0$,
\begin{enumerate}
\item[{\bf E1}] there exists a set $\Omega_r\subset X$ such that $\mu(\Omega_r^c)<r^b$;
\item[{\bf E2}] for all $x\in \Omega_r$,
$$\sup_{t\ge0}\lt|\mu_{B(x,r)}\lt(\tau_{B(x,r)}> \frac{t}{\mu(B(x,r))}\rt)-e^{-t}\rt| \leq r^a.$$
\end{enumerate}
The conditions E1-E2 imply that $\underline{\varphi}(a, \epsilon)\geq \min\{\underline{\psi}(a-\epsilon),b\} $; See Proposition~\ref{Henon} in Section~\ref{proofmainresult}.

\section{Large deviation estimates for return times for conformal repeller}\label{aprepeller}

In this section we apply our main result to conformal repellers (See \cite{Bar}).

Let $g:M \to M$ be a $C^{1+\alpha}$ map of a smooth manifold and consider a $g$-invariant compact set $J \subset M$. The map $g$ is said to be expanding on
$J$ if there exist constants $c>0$ and $\beta>1$ such that
$$\|d_xg^nv\| \geq c\beta^n \|v\|$$
for every $n \in \nds,$ $x \in J$ and $v \in T_x M.$ In addition, we call $J$ a repeller if there exists an open neighborhood $V$ of $J$ such that
$$J= \displaystyle\bigcap_{n \geq 0} g^{-n}V.$$

The map $g$ is said to be conformal on $J$ if
$$d_xg=a(x)Isom_x,$$
where $Isom_x$ denotes an isometry of the tangent space $T_x M.$

From now on, let $(J,g)$ be a conformal repeller.

Let $(\Sigma^{+}_{A}, \sigma)$ be a subshift of a finite type that defines a coding map $\chi: \Sigma^{+}_{A} \to J$ such that $\chi \circ \sigma= g\circ \chi.$
Let $\zeta$ be a H\"older continuous function on $J$ and $\mu=\mu_{\zeta}$ be the
equilibrium measure for $(g, \zeta)$. 
Let $\nu=\nu_\varphi$ be the Gibbs measure of the H\"older potential $\varphi=\zeta\circ\chi$ on $\Sigma^{+}_{A}$. Note that $\mu=\chi_*\nu$. There exists a constant $P(\varphi) \in \mathds{R}$ such that for some $\kappa_\varphi \geq 1,$ for any $x$ and $n$, we have
$$\frac{1}{\kappa_\varphi} \leq \frac{\nu(\jcal_n(x))}{exp(S_n\varphi(x)-nP(\varphi))} \leq \kappa_\varphi,$$
where $\jcal_n(x)$ is the cylinder of length $n$ containing $x$. Finally, consider the function $\psi$ such that $\log \psi=\varphi-P(\varphi).$

We collect some facts about HP-spectrum for dimensions.

\begin{Prop}[\cite{Pesin}]\label{hpspec}
For all $q\in \rds$, the following limit exists
\begin{equation}\label{specdim}
T(q)=\lim_{r \to 0} \frac{\log \int_J \mu(B(x,r))^{q-1}d\mu(x)}{-\log r}.
\end{equation}
In addition, the function $T(q)$ is real analytic for all $q \in \rds,$ $T(0)=\dim_H J,$ $T(1)=0,$ $T'(q) \leq 0$  and
$ {T}''(q) \geq 0.$ And ${T}''(q)>0$ if and only if the function $\log \psi - T'(q)\log |a(\chi(w))|$ is not cohomologous to a constant. If
and only if $\mu$ is not a measure of maximal dimension.
\end{Prop}

\begin{Rmk}
Given $q \in (-\infty, \infty),$ define $\phi_q$ on $\Sigma_{A}^{+}$ the one parameter family of functions by
$$\phi_q(w)=-T(q) \log |a(\chi(w))|+q\log \psi(w).$$ The function $T(q)$ is chosen such that $P(\phi_q)=0.$ Moreover, for any $q>1$,
$$\frac{T(q)}{1-q}=HP_\mu(q).$$
\end{Rmk}

Note that $\mu$ is exact dimensional (one can see \cite{Pesin1} for more details).

Under this context, if we consider a conformal repeller and an equilibrium state of a H\"older potential $\zeta$ we obtain a version of our principal result,
somewhat more concrete:

\begin{enumerate}
  \item in this setting we can compute the exponential rate for the dimension $\underline{\psi}$, using thermodynamic formalism;
  \item we can also estimate the exponential rate for fast return times $\underline{\varphi}$, using a technique similar to the one used to prove exponential return time statistics.
\end{enumerate}

Thus, applying our main result to this setting will give us the following theorem.

\begin{Th}\label{versionconfrep} Let $(J,g)$ be a conformal repeller and $\mu$ an equilibrium state for a H\"older potential $\zeta$. For any $\epsilon>0,$ we have:
\begin{eqnarray*}\label{eq1}
\underset{r \rightarrow 0}{\underline{\lim}} \frac{1}{\log r} \log \mu\lt(\tau_r \geq r^{-d_{\mu}- \epsilon}\rt) &\geq& g_1(\epsilon) \\
 \label{eq2}
\underset{r \rightarrow 0}{\underline{\lim}} \frac{1}{\log r} \log \mu\lt(\tau_r \leq r^{-d_{\mu}+ \epsilon}\rt) &\geq& g_2(\epsilon),
\end{eqnarray*}
where $$g_1(\epsilon)=\ds\max_{\gamma \in (0,1)}\min\lt\{(1-\gamma)\epsilon, \Lambda^{*}(-d_{\mu}-\gamma \epsilon)\rt\} >0$$ and $$g_2(\epsilon)=\ds\max_{\substack {\gamma \in (0,1) \\ \epsilon'>0 \\ \epsilon''>0}}\min \lt\{-\gamma\epsilon-\epsilon''+ \min\{d_2, \epsilon- \epsilon'\}, \Lambda^{*}(-d_{\mu}-\gamma\epsilon),\min\{a_0,\Lambda^{*}(-d_{\mu}+\epsilon')\}, \Lambda^{*}(-d_{\mu}+\epsilon'')\rt\}>0$$ with $\Lambda^{*}(x)=-x + T^{*}(x)=-x + \displaystyle\sup_{\lambda\in \rds}\{\lambda x- T(\lambda)\}$ and $a_0$, $d_2$ are some constants.
\end{Th}

\begin{Rmk}
If $\mu$ is the measure of maximal dimension the above theorem remains valid. However, since $HP_{\mu}(q)$ is constant and equal to $d_{\mu}$, it follows that
$\Lambda^{*}(x)= +\infty$ for $x \neq 0,$ which makes $g_1(\epsilon)=\ds\max_{\gamma \in (0,1)}(1-\gamma)\epsilon$ and $g_2(\epsilon)=\ds\max_{\substack{\gamma \in (0,1) \\ \epsilon'>0 \\ \epsilon''>0}}\min\{-\gamma\epsilon-\epsilon''+\min\{d_2, \epsilon-\epsilon'\},a_0\}.$
\end{Rmk}

\begin{Prop}\label{casenotmaxdim}
Suppose that $\mu$ is not the measure of maximal dimension. Then for any $\kappa<1$ and $\epsilon$ sufficiently small, $g_1 \geq \kappa c\epsilon^2$ and
$g_2 \geq \kappa c\lt(\dfrac{\epsilon}{3}\rt)^2,$ with $c=\dfrac{1}{2}(\Lambda^{*})''(-d_{\mu})=\dfrac{\lambda_{\mu}}{\sigma^2_{\mu}},$
where $\lambda_{\mu}$ is the Lyapunov exponent of $g$ and $\sigma^2_{\mu}$ the variance of $\log \psi+ \log |a|$ with respect to $\mu.$
\end{Prop}

The proof of this proposition will be done at the end of this section.

To obtain Theorem \ref{versionconfrep}, we need a fundamental theorem of large deviation theory, the Gartner-Ellis Theorem.

Let $\mu_r$ be a family of probability measures.
Consider a family $Z_r \in \rds, r \in (0,1)$ where $Z_r$ possesses the law $\mu_r$ and logarithmic moment generating function
$$\Lambda_r(\lambda)=\log \mathbb{E}\lt[e^{\lambda Z_r}\rt].$$
$\mu_r$ may satisfy the large deviation property if there exists a limit of properly scaled logarithmic moment generating functions.

\begin{Ass}\label{limlamr}
For any $\lambda \in \rds,$ the logarithmic moment generating function, defined as the limit
$$\Ll:=\lim_{n\rightarrow \infty}\frac{1}{-\log r}\Lambda_r(-\lambda\log r ),$$
exists as an extended real number. Further, the origin belongs to the interior of the interval $D_{\lambda}:=\{\lambda \in \rds; \Ll < \infty\},$
$\Lambda$ is $C^{2}$ and strictly convex.
\end{Ass}

The Fenchel-Legendre transform of $\Lambda(\lambda)$ is

$$\Lambda^{*}(x)=\sup_{\lambda \in \mathds{R}}\{\lambda x - \Lambda(\lambda)\}.$$

\begin{Rmk}
  $\Lambda^{*}$ is strictly convex and $C^1$ on its support.
\end{Rmk}
Thus, we can enunciate Gartner-Ellis Theorem (see e.g. \cite{Dembo}).
\begin{Th}[Gartner-Ellis]\label{GarEl}
If assumption \ref{limlamr} hold, then for any closed set $F$,
  \begin{eqnarray}\label{limsupge}
   \underset{r \to 0}{\overline{\lim}}\frac{1}{-\log r} \log \mu_n(F) \leq - \inf_{x \in F} \Lxx.
  \end{eqnarray}
and for any open set $G$,
   \begin{eqnarray}\label{liminfge}
  \underset{r \to 0}{\underline{\lim}}\frac{1}{-\log r}\log \mu_n(G) \geq - \inf_{x \in G } \Lxx.
  \end{eqnarray}
\end{Th}



We will apply this theorem to the family $Z_r$ defined by
$$Z_r=\frac{\log \mu(B(x,r))}{-\log r}.$$
It follows that
$$\Lambda_r(\lambda)=\log \int e^{\lambda \frac{\log \mu(B(x,r))}{-\log r}} d\mu(x).$$
Thus, from the definition of $\Lambda(\lambda),$ we get
\begin{eqnarray*}
\Lambda(\lambda) &=& \lim_{r \to 0} \frac{1}{-\log r}\log \int e^{\lambda \log \mu(B(x,r))}d\mu(x) \\
                 &=& \lim_{r \to 0} \frac{1}{-\log r}\log \int \mu(B(x,r))^{\lambda}d\mu(x).
\end{eqnarray*}

The proof of the following proposition is an immediate consequence of the Proposition \ref{hpspec}.
\begin{Prop}\label{limint}
Let $(J,g)$ be a conformal repeller and $\mu$ an equilibrium state for the H\"older potential $\zeta$. Then, for $\lambda>0$, the following limit exists
$$\Lambda(\lambda)=\lim_{r \to 0}\frac{1}{-\log r} \log \int \mu(B(x,r))^\lambda d\mu(x)=T(\lambda+1).$$
\end{Prop}
Applying Gartner-Ellis Theorem, we obtain:
\begin{Cor}\label{corldp}
Under the same conditions as in Proposition \ref{limint} we have that for all interval $I,$

$$\underset{r \rightarrow 0}{\lim}\frac{1}{-\log r} \log \mu\left(\left\{\frac{\log \mu(B(x,r))}{-\log r} \in I \right\}\right) = - \inf_{x \in I} \Lambda^{*}(x),$$

where $\Lambda^{*}(x)=-x + T^{*}(x)$ is continuous on its domain.
\end{Cor}

\begin{proof}
This equality is a direct consequence of the Theorem \ref{GarEl}.
Since the logarithmic moment generating function is defined by $\Lambda(\lambda)=T(\lambda+1)$, the Fenchel-Legendre transform of $\Lambda(\lambda)$ is
\begin{eqnarray*}
 \Lambda^{*}(x) &=& \sup_{\lambda \in \rds} \{\lambda x - \Lambda(\lambda)\}  \\
   &=& \sup_{\lambda \in \rds} \lt\{\lambda x - T(\lambda +1)\rt\}\\
   &=& \sup_{\nu \in \rds} \lt\{(\nu-1)x - T(\nu)\rt\} \\
   &=& -x + \sup_{\nu \in \rds} \lt\{\nu x - T(\nu)\rt\} \\
   &=& -x + T^{*}(x).
\end{eqnarray*}
The continuity of $\Lambda^{*}(x)$ follows from its convexity.
\end{proof}

In Figure \ref{fltlambda}, one can see a graph of the Fenchel-Legendre transform of $\Lambda$.

\begin{figure}[!htb]
  \centering
\includegraphics[width=13cm]{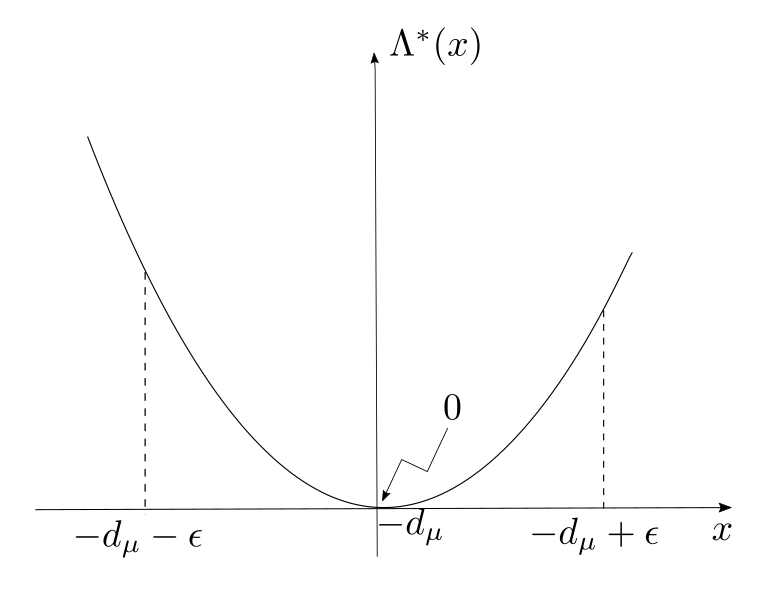}
\caption{Graph of $\Lambda^{*}.$}
\label{fltlambda}
\end{figure}

Now, one can use Corollary \ref{corldp}, to get the rate function for the dimension:
\begin{Prop}\label{ratedimrep} For any $\epsilon>0$, the exponential rate for the dimension is given by:
$$\underline{\psi}(\epsilon)= \inf_{x \in (-\infty, -d_{\mu}-\epsilon)} \Lambda^{*}(x)=\Lambda^{*}(-d_{\mu}-\epsilon)>0$$
and
$$\underline{\psi}(-\epsilon)= \inf_{x \in (-d_{\mu}+\epsilon,+\infty)} \Lambda^{*}(x)=\Lambda^{*}(-d_{\mu}+\epsilon)>0.$$
\end{Prop}
\begin{proof}
Recall that the exponential rate for dimension $\underline{\psi}$ is defined by
\begin{equation} \label{fldpsimk}
\underline{\psi}(\pm\epsilon)=\underset{r \rightarrow 0}{\underline{\lim}}\frac{1}{\log r}\log \mu \left(\left\{ \frac{\log \mu(B(x,r)) }{-\log r} \in I_{\pm\epsilon} \right\}\right),
\end{equation}

where $I_{\epsilon}=(-\infty, -d_{\mu} - \epsilon)$ and $I_{-\epsilon}=(-d_{\mu}+\epsilon, +\infty).$
So, by Corollary \ref{corldp} we have that
\begin{eqnarray*}
\underline{\psi}(\pm\epsilon)&=&\underset{r \rightarrow 0}{\underline{\lim}}\frac{1}{\log r}\log \mu \left(\left\{ \frac{\log \mu(B(x,r))}{-\log r} \in I_{\pm\epsilon} \right\}\right) \\
                                 &=& \inf_{x \in I_{\pm\epsilon}} \Lambda^{*}(x)\\
                                 &=& \Lambda^{*}(-d_{\mu}\mp\epsilon)
\end{eqnarray*}
which proves the proposition.
\end{proof}

From now on, assume that $\zeta$ is an H\"older potential such that $P(\zeta)=0.$ To obtain exponential rate for the return times (Prop. \ref{ratefastconfrep}), we recall the closing lemma.

\begin{Lem}[Closing lemma]
If $(J,g)$ is a repeller then for all $r, k, x$ such that $d(g^kx,x)<r$ there exists a point $z$ with $g^k(z)=z$ and $d(x,z)<c_1r, c_1>0.$
\end{Lem}

We will use these properties to have information on the rate function for the return times:

Given $\epsilon, \xi>0,$ define
\begin{equation}\label{Arepsilon}
A_{\epsilon}(r)=\lt\{x \in X:  \mu(B(x,r))\geq r^{d_{\mu}+\epsilon}\rt\}
\end{equation}
and
\begin{equation}\label{Armenosxi}
A_{-\xi}(r)=\lt\{x \in X:  \mu(B(x,r))\leq r^{d_{\mu}-\xi}\rt\}.
\end{equation}

\begin{Prop}\label{ratefastconfrep}
There exist constants $a_0, d_2>0$ such that for $\epsilon, \epsilon'>0,$ the exponential rate for fast return times satisfies:
  $$\underline{\varphi}(\min\{d_2, \epsilon-\epsilon'\}, \epsilon) \geq \min\{a_0, \underline{\psi}(-\epsilon')\}>0.$$
\end{Prop}

This proposition is a consequence of the following lemma.

\begin{Lem}\label{omegameasure2}
For any $d_0 \in (0,d_\mu)$ there exist constants $a_0, c_3, c_4, c_5, d_1, D >0$  and a set $\Omega_r$ such that
$$\mu(\Omega_r^c) < Dr^{a_0}$$
and for all $x_0 \in \Omega_r,$ one has
 $$\mu_{B(x_0,2r)}\lt({\tau_{B(x_0,2r)} \leq r^{-d_0}}\rt) \leq \lt(c_4-c_5\log r\rt) \mu(B(x_0, c_3r^{d_1})) + r^{-d_0}\mu(B(x_0,3r)).$$
\end{Lem}

\begin{proof}
We first claim that there exists $\Omega_r$ with $\mu(\Omega_r^c) \leq r^{a_0}$ such that for all $x_0 \in \Omega_r$ and for all
$k\leq c_0 \log \frac{1}{2r}$ we have $B(x_0,2r) \cap g^{-k}(B(x_0,2r)) = \emptyset.$

Indeed, let $c_0=\frac{d_{\mu}}{2\log m}$, where $m$ is the degree of the map $g.$
If $x_0$ is such that \linebreak $B(x_0,2r) \cap g^{-k}(B(x_0,2r)) \neq \emptyset$, there exists $x$ such that $d(x,x_0)<2r$ and $d(g^{k}x,x_0)<2r,$ thus
$d(x,g^kx)<4r.$ By the Closing lemma, there exists a point $z$ such that $g^kz=z$ and
$d(z,x)< 4c_1r.$

Define $\pcal_k=\{z: g^{k}z=z\}.$ Thus $d(x,\pcal_k)<4c_1r.$
Observe that
\begin{equation}\label{ck}
\ccal_r(k)= \lt\{x_0: B(x_0,2r) \cap g^{-k}(B(x_0,2r)) \neq \emptyset\rt\} \subset B(\pcal_k, (4c_1+2)r):= \bigcup_{y \in \pcal_k}B(y,(4c_1+2)r).
\end{equation}
Moreover, let $\xi=\frac{d_{\mu}}{4\log m}.$ Using \eqref{Armenosxi}, we have the inequality
\begin{eqnarray*}
\mu\lt(A_{(-\xi)}((8c_1+4)r) \cap B(\pcal_k, (4c_1+2)r)\rt) &\leq& \# \pcal_k \sup_{x \in A_{(-\xi)}((8c_1+4)r)} \mu (B(x,(8c_1+4)r)) \\
&\leq& m^{k} ((8c_1+4)r)^{d_{\mu}-\xi}.
\end{eqnarray*}
Take $K= c_0 \log \frac{1}{2r}$ and define
$$\Omega_r= A_{(-\xi)}((8c_1+4)r) \cap \displaystyle\bigcap_{k \leq K} B(\pcal_k, (4c_1+2)r)^c.$$
It follows from the previous inequality and Corollary \ref{corldp} that
\begin{eqnarray*}
 \mu(\Omega_r^c) &\leq& \sum_{k=1}^{K} \mu\lt(A_{(-\xi)}((8c_1+4)r) \cap B(\pcal_k,(4c_1+2)r)\rt) + \mu\lt(A_{(-\xi)}^c((8c_1+4)r)\rt) \\
  &\leq& ((8c_1+4)r)^{d_{\mu}-\xi}\sum_{k=1}^{K} m^{k}+((8c_1+4)r)^{\underline{\psi}(-\xi)-\delta}, \\
  &\leq& ((8c_1+4)r)^{d_{\mu}-\xi}m^{K+1} + ((8c_1+4)r)^{\underline{\psi}(-\xi)-\delta}  \\
  &\leq& D r^{a_0},
\end{eqnarray*}
for $\delta >0$ sufficiently small and $a_0=\min\{d_{\mu}-\xi - c_0\log m, \underline{\psi}(-\xi)-\delta\}.$
We observe that $x_0 \in \Omega_r$ implies that $x_0 \notin B(\pcal_k, (4c_1+2)r)$ for all $k \leq c_0\log\frac{1}{2r}$, therefore from \eqref{ck},
$B(x_0,2r) \cap g^{-k}(B(x_0,2r)) = \emptyset$ which proves our initial claim.

We will now estimate the quantity $\mu_{B(x_0,2r)}\lt(g^{-k}B(x_0,2r)\rt)$ for large values of $k$.

Recall that $\zeta$ is a H\"older potential such that $P(\zeta)=0$ and that the Ruelle-Perron-Frobenius operator \linebreak $\lcal_\zeta: C(M) \to C(M)$ is defined
by $$\lcal_\zeta(f)(x)=\sum_{y \in g^{-1}(x)}e^{\zeta(y)}f(y),$$ for $f \in C(M)$ and $x \in M.$
By induction, for every $n \geq 1$,
$$\lcal^n_\zeta(f)(x)=\sum_{y \in g^{-n}(x)}e^{S_n\zeta(y)}f(y),$$
where $S_n\zeta= \displaystyle\sum_{k=0}^{n-1} \zeta\circ g^k.$
Now we have that
\begin{eqnarray*}
   \mu\lt(B(x_0,2r) \cap g^{-k}B(x_0,2r)\rt) &=& \int \mathds{1}_{B(x_0,2r)}\mathds{1}_{B(x_0,2r)}\circ g^k d\mu \\
   &=& \int \lcal^k(\mathds{1}_{B(x_0,2r)})\mathds{1}_{B(x_0,2r)} d\mu \\
   &\leq& \mu({B(x_0,2r)}) \left\|\lcal^k(\mathds{1}_{B(x_0,2r)})\right\|_{\infty}.
\end{eqnarray*}
Hence,
\begin{equation}\label{mblk}
  \mu_{{B(x_0,2r)}}\lt(g^{-k}{B(x_0,2r)}\rt) \leq \left\|\lcal^k(\mathds{1}_{B(x_0,2r)})\right\|_{\infty}.
\end{equation}

It follows that for all $R \in \jcal_k,$
\begin{eqnarray*}
\lcal^k(\mathds{1}_R)(x)  &=& \sum_{y \in g^{-k}x}e^{S_k \zeta(y)}\mathds{1}_R(y) \\
                          &\leq & \sum_{y \in g^{-k}x, \ y \in R}k_{\zeta} \mu(R),
\end{eqnarray*}
where the last inequality follows from the Gibbs property since $P(\zeta)=0.$ In addition, the preimage of
$x$ under $g^k$ has just one element in $R$. Then,
\begin{eqnarray}\label{lk1j}
\lcal^k(\mathds{1}_R)(x) \leq k_{\zeta} \mu(R).
\end{eqnarray}
We have
\begin{eqnarray*}
  \lcal^k(\mathds{1}_{B(x_0,2r)}) &=& \sum_{R \in \jcal_k, R\cap B(x_0,2r) \neq \emptyset} \lcal^k(\mathds{1}_R)\\
  &\leq& \sum_{R \in \jcal_k, R \subset B(x_0,2r+ \text{diam} (\jcal_k))}k_{\zeta} \mu(R) \\
  &\leq& k_{\zeta} \mu(B(x_0,2r+ \text{diam} (\jcal_k))).
\end{eqnarray*}
This implies by \eqref{mblk}
$$\mu_{B(x_0,2r)}\lt(g^{-k}{B(x_0,2r)}\rt) \leq  k_{\zeta} \mu(B(x_0,2r+ \text{diam} (\jcal_k))).$$
Let $k>c_0 \log \frac{1}{2r}$. We have $\text{diam}(\jcal_k)<c_2\beta^{-k}.$ For $k$ such that $c_2\beta^{-k} >r$, we have
$$\beta^{-k}< \beta^{-c_0 \log \frac{1}{2r}}=(2r)^{c_0 \log \beta},$$ which implies
\begin{eqnarray*}
\mu(B(x_0,2r+ \text{diam} (\jcal_k))) & \leq & \mu(B(x_0,3c_2\beta^{-k})) \\
                               & \leq & \mu(B(x_0,c_3r^{c_0 \log \beta})).
\end{eqnarray*}
When $k$ satisfies $c_2\beta^{-k} \leq r$, we obtain
\begin{eqnarray*}
\mu(B(x_0,2r+ \text{diam} (\jcal_k))) & \leq & \mu(B(x_0,3c_2\beta^{-k})) \\
                               & \leq & \mu(B(x_0,3r)).
\end{eqnarray*}

Recall that for all $x_0 \in \Omega_r$ and for all $k\leq c_0 \log \frac{1}{2r}, \ B(x_0,2r) \cap g^{-k}(B(x_0,2r)) = \emptyset.$ We get
\begin{eqnarray*}
 & & \mu_{B(x_0,2r)}\lt({\tau_{B(x_0,2r)} \leq r^{-d_0}}\rt)\\&\leq& \sum_{k=1}^{r^{-d_0}} \mu_{B(x_0,2r)}\lt(g^{-k}B(x_0,2r)\rt) \\
    &=& \sum_{k=c_0 \log \frac{1}{2r}}^{\lfloor \frac{\log c_2 - \log r}{\log \beta}\rfloor} \mu_{B(x_0,2r)}\lt(g^{-k}B(x_0,2r)\rt)+ \sum_{k= \lfloor \frac{\log c_2 - \log r}{\log \beta}\rfloor +1}^{r^{-d_0}} \mu_{B(x_0,2r)}\lt(g^{-k}B(x_0,2r)\rt) \\
    & \leq & \sum_{k=c_0 \log \frac{1}{2r}}^{\lfloor \frac{\log c_2 - \log r}{\log \beta}\rfloor} \mu(B(x_0, c_3r^{d_1}))+ \sum_{k= \lfloor \frac{\log c_2 - \log r}{\log \beta}\rfloor +1}^{r^{-d_0}} \mu(B(x_0,3r))\\
    &\leq& \lt(c_4-c_5\log r\rt) \mu(B(x_0, c_3r^{d_1})) + r^{-d_0}\mu(B(x_0,3r))
\end{eqnarray*}
with $d_1=c_0\log\beta$, which ends the proof.
\end{proof}

From the theory of conformal repellers we obtain the following lemma:

\begin{Lem}\label{auxlem}
There exists $d_3>0$ such that for all $r,$ $\mu(B(x,r)) \leq r^{d_3}.$
\end{Lem}
\begin{proof}
  One can use (6.17) in \cite{Bar}.
\end{proof}

\begin{proof}[Proof of Proposition \ref{ratefastconfrep}]
Using the above lemma we have that there exist constants \linebreak $c_6, d_2>0$ such that $\lt(c_4-c_5\log r\rt) \mu(B(x_0, c_3r^{d_1})) \leq c_6r^{d_2},$ for all $x_0.$ Let $0<\epsilon'<\epsilon,$
for $x \in \Omega_r \cap A_{(- \epsilon')}(3r),$ using Lemma \ref{omegameasure2} we obtain
\begin{eqnarray*}
\mu_{B(x_0,2r)}\lt({\tau_{B(x_0,2r)} \leq r^{-d_0}}\rt) & \leq & c_6r^{d_2}+r^{-d_0}(3r)^{d\mu-\epsilon'} \\
     &\leq & c_7r^{\min\{d_2,-d_0+d_{\mu}-\epsilon'\}}.
\end{eqnarray*}
Hence
\begin{eqnarray*}
\mu\lt(\lt\{x: \mu_{B(x_0,2r)}\lt({\tau_{B(x_0,2r)} \leq r^{-d_0}}\rt) > c_7r^{\min\{d_2,-d_0+d_{\mu}-\epsilon'\}}\rt\}\rt)  & \leq & \mu\lt(\lt(\Omega_r \cap A_{(- \epsilon')}(3r)\rt)^c\rt)\\
  & \leq & \mu(\Omega_r^c)+ \mu\lt(A_{(- \epsilon')}^c(3r)\rt).
\end{eqnarray*}
Take $d_0=d_{\mu}-\epsilon.$
Finally, using Lemma \ref{esslem}, we get
$$\underline{\varphi}(\min\{d_2, \epsilon- \epsilon'\}, \epsilon) \geq \min\{a_0, \underline{\psi}(-\epsilon')\}$$
and the proposition is proved.

\end{proof}

We are now able to prove Theorem \ref{versionconfrep}
\begin{proof}[Proof of the Theorem \ref{versionconfrep}] For $\gamma \in (0,1),$ we get by Proposition \ref{ratedimrep} that

$$\underline{\psi}(\gamma\epsilon)= \Lambda^{*}(-d_{\mu}-\gamma\epsilon)>0$$ and $$\underline{\psi}(-\epsilon'')= \Lambda^{*}(-d_{\mu}+\epsilon'')>0.$$
Moreover, by Proposition \ref{ratefastconfrep}, $$\underline{\varphi}(\min\{d_2, \epsilon- \epsilon'\}, \epsilon) \geq \min\{a_0, \underline{\psi}(-\epsilon')\}=\min\{a_0, \Lambda^{*}(-d_{\mu}+\epsilon')\}>0.$$
Thus, it follows from Theorem \ref{principalth}, that
\begin{eqnarray*}
\underset{r \rightarrow 0}{\underline{\lim}} \frac{1}{\log r} \log \mu\lt(\tau_r \geq r^{-d_{\mu}- \epsilon}\rt) &\geq&  \max_{\gamma \in (0,1)}\min\{(1-\gamma)\epsilon, \Lambda^{*}(-d_{\mu}-\gamma \epsilon)\} >0
\end{eqnarray*}
and
\begin{eqnarray*}
&& \underset{r \rightarrow 0}{\underline{\lim}} \frac{1}{\log r} \log \mu\lt(\tau_r \leq r^{-d_{\mu}+ \epsilon}\rt) \\
&\geq& \max_{\substack{\gamma \in (0,1) \\ \epsilon'>0 \\ \epsilon''>0}}\min \lt\{-\gamma\epsilon-\epsilon''+\min\{d_2, \epsilon-\epsilon'\}, \Lambda^{*}(-d_{\mu}-\gamma \epsilon),\min\{a_0, \Lambda^{*}(d_{\mu}+\epsilon')\},\Lambda^{*}(-d_{\mu}+\epsilon'')\rt\}>0,
\end{eqnarray*}
thus the theorem is proved.
\end{proof}

\begin{proof}[Proof of the Proposition \ref{casenotmaxdim}]
By definition $T^{*}(x)=\displaystyle\sup_{q \in \rds}\{qx- T(q)\}.$ The supremum is achieved for $q$ such that $\dfrac{d}{dq}(qx-T(q))=x- T'(q)=0,$ that is, $T'(q)=x.$
Thus, for any $q \in \rds$
$$T^{*}(T'(q))=qT'(q)- T(q).$$
So, it follows that
$$(T^{*})'(T'(q))T''(q)=qT''(q),$$
and hence, $(T^{*})'(T'(q))=q$ and, differentiating, we obtain
$$(T^{*})''(T'(q))T''(q)=1,$$ that is
$$(T^{*})''(T'(q))=\frac{1}{T''(q)} \ \mbox{for every} \ q \in \rds \ \mbox{such that} \ T'(q)=x.$$

Since $\Lambda(\lambda)=T(\lambda+1)$ and $\Lambda^{*}(x)=-x+T^{*}(x)$ we conclude that
$$(\Lambda^{*})''(x)= \frac{1}{T''(q)}, \ \mbox{where} \ T'(\lambda+1)=x.$$
Moreover, by Proposition \ref{hpspec} this is non negative.

For $x=-d_{\mu}$ we have $\lambda=0.$
Then, by Lemma $5$ in \cite{Pesin}
$$(\Lambda^{*})''(-d_{\mu})= \frac{1}{T''(1)}=\frac{\lambda_{\mu}}{\sigma^2_{\mu}}.$$
Finally, for $\epsilon$ sufficiently small, $g_1 \geq \kappa c\epsilon^2$
and taking $\epsilon'=\epsilon''=\gamma\epsilon=\dfrac{\epsilon}{\eta},$ with $\eta>3$, we have
$g_2 \geq  \kappa c \lt(\dfrac{\epsilon}{3}\rt)^2.$

\end{proof}


\section{Proof of the main result}\label{proofmainresult}
In this section we prove the Theorem \ref{principalth} using the method developed in \cite{Saussol}.
We begin by the following lemma which will be needed in the proof of our main theorem.
\begin{Lem}\label{esslem}
Let $(a_i(r))_{i=1,...,p}, a_i(r)>0.$ If $\gamma_i= \underset{r \rightarrow 0}{\underline{\lim}} \frac{1}{\log r}\log a_i(r)>0$. Then
$$\underset{r \rightarrow 0}{\underline{\lim}} \frac{1}{\log r}\log \left(\sum_{i=1}^{p}a_i(r)\right) \geq \min_{i=1,...,p}\gamma_i.$$
\end{Lem}

\begin{proof}
For all $\epsilon>0$ there exists $r_i>0$ such that $r<r_i$ implies $a_i \leq r^{\gamma_i-\epsilon}.$
Let $\epsilon>0$ sufficiently small such that $\gamma_i-\epsilon>0.$ We have,
$$\sum_{i=1}^{p}a_i(r) \leq \sum_{i=1}^{p}r^{\gamma_i-\epsilon}\leq pr^{\min\{\gamma_i\}-\epsilon}$$
and this implies
$$\frac{1}{\log r}\left(\sum_{i=1}^{p}a_i(r)\right)\geq \min_{i=1,...,p}\{\gamma_i\}-\epsilon + \frac{\log p}{\log r}.$$
Finally,
$$\underset{r \rightarrow 0}{\underline{\lim}} \frac{1}{\log r}\log \left(\sum_{i=1}^{p}a_i(r)\right) \geq \min_{i=1,...,p}\{\gamma_i\} - \epsilon.$$
The result is proved since $\epsilon$ can be chosen arbitrarily small.
\end{proof}

Denote
$$\underline{f}(\epsilon)= \underset{r \rightarrow 0}{\underline{\lim}} \frac{1}{\log r} \log \mu\lt(\tau_r \geq r^{-d_{\mu}- \epsilon}\rt)$$
and
$$\underline{f}(-\epsilon)= \underset{r \rightarrow 0}{\underline{\lim}} \frac{1}{\log r} \log \mu\lt(\tau_r \leq r^{-d_{\mu}+ \epsilon}\rt).$$

\begin{proof}[Proof of the Theorem \ref{principalth}]
Let $\gamma \in (0,1).$ We have

\begin{eqnarray*}
  \mu(\{x: \tau_r(x) \geq r^{-d_{\mu}-\epsilon}\}) &\leq& \mu\lt(\lt\{x \in A_{\gamma\epsilon}\lt(\frac{r}{4}\rt): \tau_r(x) \geq r^{-d_{\mu}-\epsilon}\rt\}\rt)  \\
   &+& \mu\lt(\lt\{x \in A_{\gamma\epsilon}^c\lt(\frac{r}{4}\rt): \tau_r(x) \geq r^{-d_{\mu}-\epsilon}\rt\}\rt).
\end{eqnarray*}

Let us define the set $$M_r=\lt\{x \in A_{\gamma\epsilon}\lt(\frac{r}{4}\rt): \tau_r(x) \geq r^{-d_{\mu}-\epsilon}\rt\}.$$
Let $\lt(B\left(x_i, \frac{r}{2}\right)\rt)_i$ be a family of balls of radius $r/2$ centered at points of $A_{\gamma\epsilon}(\frac{r}{4})$ that covers $M_r$ and such that
$B\left(x_i, \frac{r}{4}\right)\cap B\left(x_j,\frac{r}{4}\right)=\emptyset \ if \ x_i \neq x_j.$ We have
\begin{eqnarray*}
  \mu\lt(\lt\{x: \tau_r(x) \geq r^{-d_{\mu}-\epsilon}\rt\}\rt) &\leq& \mu(\cup_i B_i \cap M_r)+ \mu\lt(\lt\{x \in A_{\gamma\epsilon}^c\lt(\frac{r}{4}\rt): \tau_r(x) \geq r^{-d_{\mu}-\epsilon}\rt\}\rt) \\
  &\leq& \sum_i \mu(B_i \cap M_r)+ \mu\lt(A_{\gamma\epsilon}^c\lt(\frac{r}{4}\rt)\rt).
\end{eqnarray*}
Using first the triangle inequality and then Kac's lemma and Markov inequality, we obtain
$$\mu(B_i \cap M_r) \leq \mu\lt(B_i \cap \lt\{\tau_{B_i} \geq r^{-d_{\mu}-\epsilon}\rt\}\rt)\leq r^{d_{\mu}+ \epsilon} \int_{B_i} \tau_{B_i} d\mu = r^{d_{\mu}+ \epsilon}.$$
Observe that $\displaystyle\sum_i \left(\frac{r}{4}\right)^{d_{\mu}+\gamma\epsilon} \leq \sum_i \mu\lt(B\left(x_i, \frac{r}{4}\rt)\right) \leq 1.$ Thus, since the
balls are disjoint it follows that the number of balls is bounded by $\left(\frac{1}{4}r\right)^{-d_{\mu}-\gamma\epsilon}$.
Therefore,
\begin{eqnarray*}
\mu\lt(\lt\{x: \tau_r(x) \geq r^{-d_{\mu}-\epsilon}\rt\}\rt) &\leq& \sum_i r^{d_{\mu} +\epsilon} + \mu\left(A_{\gamma\epsilon}^c\left(\frac{r}{4}\right)\right)\\
 &\leq& \left(\frac{1}{4}r\right)^{-d_{\mu}-\gamma\epsilon}r^{d_{\mu}+\epsilon} +  \mu\left(A_{\gamma\epsilon}^c\left(\frac{r}{4}\right)\right)\\
 &\leq& 4^{d_{\mu}+\gamma\epsilon} r^{(1-\gamma)\epsilon} +  \mu\left(A_{\gamma\epsilon}^c\left(\frac{r}{4}\right)\right).
\end{eqnarray*}
Thus,
$$\frac{1}{\log r} \log \mu\lt(\lt\{x: \tau_r(x) \geq r^{-d\mu-\epsilon}\rt\}\rt) \geq \frac{1}{\log r} \log \lt(4^{d_{\mu}+\gamma\epsilon} r^{(1-\gamma)\epsilon} +  \mu\left(A_{\gamma\epsilon}^c\left(\frac{r}{4}\right)\right)\rt).$$
Hence, by Lemma \ref{esslem}, we get
\begin{eqnarray*}
\underline{f}(\epsilon) &\geq& \underset{r \rightarrow 0}{\underline{\lim}} \frac{1}{\log r} \log \left(4^{d_{\mu}+\gamma\epsilon} r^{(1-\gamma)\epsilon} +  \mu\lt(A_{\gamma\epsilon}^c\lt(\frac{r}{4}\rt)\rt)\right) \\
   &\geq& \min \lt\{\underset{r \rightarrow 0}{\underline{\lim}} \frac{1}{\log r} \log \left(4^{d_{\mu}+\gamma\epsilon} r^{(1-\gamma)\epsilon}\right), \underset{r \rightarrow 0}{\underline{\lim}} \frac{1}{\log r} \log \mu\left(A_{\gamma\epsilon}^c\left(\frac{r}{4}\right)\right)\rt\} \\
&=& \min \left\{(1-\gamma)\epsilon, \underline{\psi}(\gamma\epsilon) \right\}.
\end{eqnarray*}
This proves the first statement.

Now, let $\epsilon''>0$. We define
$$\Gamma_{r}=\lt\{x \in A_{\gamma\epsilon}(2r) \cap  A_{(-\epsilon'')}(2r): \tau_r(x) \leq r^{-d_{\mu}+\epsilon}\rt\}$$ and
$$D_r=\lt\{x_0: \mu_{B(x_0,2r)}(\tau_{B(x_0,2r)} \leq r^{-d_{\mu}+\epsilon}) \leq Cr^a\rt\}, \ C>0.$$

Let $(B(x_i, 2r))_i$ be a family of balls of radius $2r$ centered at points of \linebreak
$A_{\gamma\epsilon}(2r)\cap D_r \cap A_{(-\epsilon'')}(2r)$ that covers $\Gamma_{r}\cap D_r$ and such that
$B\left(x_i, r\right)\cap B\left(x_j,r\right)=\emptyset \ if \ x_i \neq x_j.$
We have
\begin{eqnarray*}
\mu\lt(\lt\{x: \tau_r(x) \leq r^{-d_{\mu}+\epsilon}\rt\}\rt) &\leq& \mu\lt(\lt\{x \in A_{\gamma\epsilon}(2r) \cap D_r \cap  A_{(-\epsilon'')}(2r): \tau_r(x) \leq r^{-d_{\mu}+\epsilon}\rt\}\rt) \\
  & &+  \mu\lt(\lt\{x \in (A_{\gamma\epsilon}(2r) \cap D_r\cap  A_{(-\epsilon'')}(2r))^c: \tau_r(x) \leq r^{-d_{\mu}+\epsilon}\rt\}\rt) \\
  &\leq& \mu\lt(\cup_iB(x_i, 2r) \cap \Gamma_{r}\cap D_r\rt)+ \mu\lt(A_{\gamma\epsilon}^c(2r)\rt) + \mu\lt(D_r^c\rt) + \mu\lt( A_{(-\epsilon'')}^c(2r)\rt).
\end{eqnarray*}
We remark that
\begin{eqnarray*}
\mu\lt(\cup_iB(x_i, 2r)\cap \Gamma_r \cap D_r\rt) &\leq&  \sum_{i}\mu(B(x_i, 2r) \cap \Gamma_r \cap D_r) \\
&\leq& \sum_{i}\mu(B(x_i, 2r)) \frac{1}{\mu(B(x_i, 2r))}\mu\lt(B(x_i, 2r) \cap \lt\{\tau_{B(x_i,2r)}\leq r^{-d_{\mu}+\epsilon}\rt\}\rt)\
\end{eqnarray*}
where the last inequality follows from $\lt\{\tau_{B(x_i,r)} \leq r^{-d_{\mu}+\epsilon}\rt\} \subset \lt\{\tau_{B(x_i,2r)} \leq r^{-d_{\mu}+\epsilon}\rt\}.$
Therefore, by definition of $D_r$,
\begin{eqnarray*} \label{eq*}
& &\mu\lt(\lt\{x: \tau_r(x) \leq r^{-d_{\mu}+\epsilon}\rt\}\rt) \\&\leq& \sum_{i}\mu(B(x_i, 2r)) \mu_{(B(x_i, 2r))}\lt(\tau_{B(x_i,2r)} \leq r^{-d_{\mu}+\epsilon}\rt) + \mu\lt(A_{\gamma\epsilon}^c(2r)\rt) + \mu\lt(D_r^c\rt)
   + \mu\lt( A_{(-\epsilon'')}^c(2r)\rt) \\
 &\leq& \sum_{i}\mu(B(x_i, 2r))Cr^a + \mu\lt(A_{\gamma\epsilon}^c(2r)\rt) + \mu\lt(D_r^c\rt) + \mu\lt( A_{(-\epsilon'')}^c(2r)\rt).
\end{eqnarray*}
Observe that $\displaystyle\sum_{i}r^{d_{\mu}+\gamma\epsilon} \leq \sum_{i} \mu(x_i,r) \leq 1.$
Thus, since balls are disjoint it follows that the number of balls is bounded by $r^{-d_{\mu}-\gamma\epsilon}$ and
\begin{eqnarray*}
\sum_{i}\mu(B(x_i, 2r)) &\leq& \sum_{i} (2r)^{d_{\mu}-\epsilon''} \\
                        &\leq& r^{-d_{\mu}-\gamma\epsilon}(2r)^{d_{\mu}-\epsilon''} \\
                        &\leq& 2^{d_{\mu}-\epsilon''}r^{-\gamma\epsilon -\epsilon''}.
\end{eqnarray*}
Then, we obtain that

$$\mu\lt(\lt\{x: \tau_r(x) \leq r^{-d_{\mu}+\epsilon}\rt\}\rt) \leq C2^{d_\mu-\epsilon''}r^{-\gamma\epsilon-\epsilon'' +a} + \mu\lt(A_{\gamma\epsilon}^c(2r)\right)+\mu\lt(D_r^c\rt)+ \mu\left( A_{(-\epsilon'')}^c(2r)\right).$$

Hence,
$$\underline{f}(-\epsilon) \geq \underset{r \rightarrow 0}{\underline{\lim}} \frac{1}{\log r} \log \left (C2^{d_{\mu}-\epsilon''}r^{-\gamma\epsilon-\epsilon''+a} + \mu\lt(A_{\gamma\epsilon}^c(2r)\right)+\mu\lt(D_r^c\rt)+ \mu\left( A_{(-\epsilon'')}^c(2r)\right)\right).$$
Finally, using the definitions of $\underline{\psi}$ and $\underline{\varphi}$ we get by Lemma \ref{esslem} that

$$\underline{f}(-\epsilon) \geq \min \left\{-\gamma\epsilon-\epsilon''+a, \underline{\psi}(\gamma\epsilon),\underline{\varphi}(a,\epsilon), \underline{\psi}(-\epsilon'')\right\}.$$
This concludes the proof of the theorem.
\end{proof}

We finish with a brief result that may help to estimate the rate function for fast returns.
\begin{Prop}\label{Henon}
If there exist constants $a,b>0$ such that for all $r \in (0,1):$
\begin{itemize}
  \item there exists a set $\Omega_r$ such that $$\mu(\Omega_r^c)<r^b;$$
  \item for all $x \in \Omega_r,$
$$\lt|\mu_{B(x,r)}\lt(\tau_{B(x,r)}> \frac{t}{\mu(B(x,r))}\rt)-e^{-t}\rt| \leq r^a,$$ for every $t>0$.
\end{itemize}
Then, $\underline{\varphi}(a, \epsilon)\geq \min\{\underline{\psi}(a-\epsilon),b\}.$
\end{Prop}

\begin{proof}
Take $t=Cr^a, C>0.$ Making the first order expansion of $e^{-t}$, we have for $x\in\Omega_r$
$$\lt|\mu_{B(x,r)}\lt(\tau_{B(x,r)}> \frac{Cr^a}{\mu(B(x,r))}\rt)-1+Cr^a + o(r^{2a})\rt| \leq r^a,$$ which implies
$$\lt|\mu_{B(x,r)}\lt(\tau_{B(x,r)}< \frac{Cr^a}{\mu(B(x,r))}\rt)+Cr^a + o(r^{2a})\rt| \leq r^a.$$ So, it follows that
$$\mu_{B(x,r)}\lt(\tau_{B(x,r)}< \frac{Cr^a}{\mu(B(x,r))}\rt) < r^a.$$
Let $N_r$ be a set defined by
$N_r=\lt\{x: \mu (B(x,r))\geq r^{d_{\mu}+a - \epsilon}\rt\}.$ For $x \in N_r\cap\Omega_r$ we obtain
$$\mu_{B(x,r)}\lt(\tau_{B(x,r)}< Cr^{-d_{\mu}+\epsilon}\rt) < r^a.$$
Thus,
$$\mu\lt(\lt\{x: \mu_{B(x,2r)}\lt(\tau_{B(x,2r)}\leq r^{-d_{\mu}+ \epsilon}\rt)>2^ar^a \rt\}\rt) \leq \mu((N_{2r}\cap\Omega_{2r})^c)\leq \mu(N_{2r}^c)+ \mu(\Omega_{2r}^c).$$
Finally, by Lemma \ref{esslem}, we get
$$\underline{\varphi}(a, \epsilon) \geq \min\{\underline{\psi}(a-\epsilon),b\}.$$
\end{proof}

\end{document}